\input ./style/arxiv-vmsta.cfg
\documentclass[numbers,compress,v1.0.1]{vmsta}

\usepackage{enumitem}

\volume{4}% Updated by VTEXPTS2LaTeX.exe, 26.10.2017 16:00
\issue{3}% Updated by VTEXPTS2LaTeX.exe, 26.10.2017 16:00
\pubyear{2017}
\firstpage{199}% Updated by VTEXPTS2LaTeX.exe, 26.10.2017 16:00
\lastpage{217}% Updated by VTEXPTS2LaTeX.exe, 26.10.2017 16:00
\doi{10.15559/17-VMSTA83}% Updated by VTEXPTS2LaTeX.exe, 24.08.2017
\startlocaldefs

\newcommand{\rrvert}{\vert}
\newcommand{\llvert}{\vert}
\allowdisplaybreaks

\newtheorem{thm}{Theorem}[section]
\newtheorem{lem}{Lemma}[section]

\theoremstyle{definition}
\newtheorem{zau}{Remark}[section]
\newtheorem{ozn}{Definition}[section]
\newtheorem{ex}{Example}[section]

\DeclareMathOperator{\pr}{\mathsf P}
\DeclareMathOperator{\M}{\mathsf E}

\newcommand{\R}{\mathbb{R}}
\newcommand{\N}{\mathbb{N}}
\endlocaldefs

\begin{document}
\begin{frontmatter}

%\title{}
%\author[]{\inits{}\fnm{}\snm{}\corref{cor1}}\email{}
%\cortext[cor1]{Corresponding author.}
%
%\author[]{\inits{}\fnm{}\snm{}}\email{}%\orcid{0000-0000-0000-0000}
%
%%\fnref{f1}
%%\fntext[]{Some remarks}
%
%\address[]{}
%\address[]{}
%
%
%\begin{abstract}
%\end{abstract}

\title{Asymptotic behavior of functionals of the solutions to~inhomogeneous
It\^{o} stochastic differential equations with nonregular dependence on
parameter}

\author{\inits{G.}\fnm{Grigorij}\snm{Kulinich}}\email{zag\_mat@univ.kiev.ua}
\author{\inits{S.}\fnm{Svitlana}\snm{Kushnirenko}\corref{cor1}}\email{bksv@univ.kiev.ua}
\cortext[cor1]{Corresponding author.}
\address{Taras Shevchenko National University of Kyiv, 64/13,
Volodymyrska Street, 01601~Kyiv, Ukraine}

\markboth{G. Kulinich, S. Kushnirenko}{Asymptotic behavior of functionals}

\begin{abstract}
The asymptotic behavior, as $T\to\infty$, of some functionals of the form
$I_T(t)=F_T(\xi_T(t))+\int_{0}^{t} g_T(\xi_T(s))\, dW_T(s)$, $
t\ge0$ is studied. Here $\xi_T(t)$ is the solution to the
time-inhomogeneous It\^{o} stochastic differential equation
\[
d\xi_T (t)=a_T \bigl(t,\xi_T(t) \bigr)
\,dt+dW_T(t),\quad t\ge0,\ \xi_T (0)=x_0,
\]
$T>0$ is a parameter, $a_T (t,x), x\in\mathbb{R}$ are measurable functions,
$ |a_T(t,x)| \leq C_T$ for all $x\in\mathbb{R}$ and $ t\ge0$, $W_T(t)$ are
standard Wiener processes, $F_T(x), x\in\mathbb{R}$ are continuous functions,
$g_T(x), x\in\mathbb{R}$ are measurable locally bounded functions, and everything
is real-valued. The explicit form of the limiting processes for $I_T(t)$ is
established under nonregular dependence of $a_T (t,x)$ and $g_T(x)$ on the
parameter $T$.
\end{abstract}

%
% Keywords separated by \sep
% Key1\sep Key2
%
\begin{keywords}[2010]
\kwd{60H10} \kwd{60F17} \kwd{60J60}\vspace*{24pt}
\end{keywords}

%\begin{keywords}
%\kwd{}
%\kwd{}
%\kwd{}
%\kwd{}
%\end{keywords}
%\begin{keywords}[2010]% [PACS], [JEL]
%\kwd{}
%\kwd{}
%\kwd{}
%\kwd{}
%\end{keywords}

\received{29 May 2017}% Updated by VTEXPTS2LaTeX.exe, 24.08.2017 12:34
\revised{6 August 2017}% Updated by VTEXPTS2LaTeX.exe, 24.08.2017 12:34
\accepted{10 August 2017}% Updated by VTEXPTS2LaTeX.exe, 24.08.2017 12:34
\publishedonline{22 September 2017}
\end{frontmatter}

%s1 ###
\section{Introduction}

Consider the time-inhomogeneous It\^{o} stochastic differential equation
%
%e1 ###
\begin{equation}
\label{o1} d\xi_T (t)=a_T \bigl(t,
\xi_T(t) \bigr)\,dt+dW_T(t),\quad t\ge0,\
\xi_T (0)=x_0,
\end{equation}
where $T>0$ is a parameter, $a_T (t,x), x\in\R$ are real-valued measurable
functions such that $ \llvert a_T(t,x) \rrvert \leq L_T$ for all $(t,x)$ and some
family of constants $L_T>0$, and $W_T=\{W_T(t), t\geq 0\}, T>0$ is a family
of standard Wiener processes defined on a complete probability space
$(\varOmega,\Im, \pr)$.

It is known from Theorem~4 in \cite{paper1} that for any $T>0$ and $x_0\in\R$
equation~(\ref{o1}) possesses a unique strong solution $\xi _T=\{\xi_T(t),
t\geq0\}$.

In this paper, we study the weak convergence, as $T\to\infty$, of the
processes $I_T(t)=F_T(\xi_T(t))+\int_{0}^{t} g_T(\xi_T(s))\, dW_T(s)$,
where $\xi_T(t)$ is the solution to stochastic differential equation~(\ref{o1}), $ F_T (x) $ is a family of continuous real-valued functions, $
g_T (x) $ is a family of measurable, locally bounded real-valued functions.
All the results about asymptotic behavior are obtained under the condition
which provides certain proximity
%(in the defined below terms), as $T\to\infty$,
of the coefficients $a_T (t,x)$ of equation~(\ref{o1}) to some measurable
functions $\hat{a}_T (x)$. In such situation, the limit processes, obtained
under condition $T\to\infty$, are some functionals of the limits of the
solutions $\hat{\xi}_T(t)$ to the homogeneous stochastic differential
equations
%
%e2 ###
\begin{equation}
\label{o2} d\hat{\xi}_T (t)=\hat{a}_T \bigl(\hat{
\xi}_T(t) \bigr)\,dt+dW_T(t).
\end{equation}

The present paper generalizes similar results from \cite{paperL} for the
unique strong solutions $\hat{\xi}_T$ to homogeneous stochastic differential
equations~(\ref{o2}) to the case for the solutions ${\xi }_T(t)$ to
inhomogeneous equations~(\ref{o1}). Some results about solutions
$\hat{\xi}_T$ to homogeneous equations~(\ref{o2}), which obtained in
\cite{paperL}, have been extended to solutions $\xi_T$ to inhomogeneous
equations~(\ref{o1}). Under the certain proposed conditions, which present a
novelty in comparison with \cite{paperL}, we prove that the asymptotic
behavior of the solutions and some functionals of the solutions to
inhomogeneous It\^{o} stochastic differential equations~(\ref{o1}) is the
same as that for the solutions to homogeneous It\^{o} stochastic differential
equations~(\ref{o2}). The present paper also complements results of paper
\cite{paper10}. Moreover, we assume that the drift coefficient $a_T (t,x)$ in
equation~(\ref{o1}) can have nonregular dependence on the parameter $T$. For
example, the drift coefficient $a_T (t,x)$ can tend, as $T\to \infty$, to
infinity
% in some sense
at some points $x_k$ and at some points $t_k$ as well, or it can have
degeneracies of some other types. Such a nonregular dependence on $T$ of the
coefficients in equation~(\ref{o1}) appeared for the first time in \cite
{paper4} and \cite{paper3}, where the limit behavior of the normalized
unstable solution of It\^{o} stochastic differential equation, as $t\to\infty$,
was investigated for homogeneous equations. In those papers, a special
dependence of the coefficients $a_T(x)=\sqrt{T}a(x\sqrt{T})$ on the parameter
$T$ was considered with $a(x)\in L_1 (\mathbb{R} )$. The special
dependence of the coefficients $a_T (t,x)=\sqrt{T}a(tT,x\sqrt{T})$ on the
parameter $T$ was considered in \cite{paper44} for inhomogeneous stochastic
differential equations~(\ref{o1}).

A more detailed review of the known results in this area is presented, for
example, in~\cite{paper88} and \cite{paperL}.

The paper is organized as follows. In Section~2, we set the notations and
formulate basic definitions. Section~3 contains the statements of the main
results. In Section~4, they are proved. Auxiliary results are collected in
Section~5. Section~6 gives examples.

%s2 ###
\section{Main definitions}

In what follows we denote by $C,L,N,C_N,L_N$ any constants that do
not depend on $T$, $x$ and $t$. To formulate and prove the main results we
introduce the functions of the form
%
%e3 ###
\begin{equation}
\label{o3} f_T(x)=\int_{0}^{x}\exp \biggl\{-2
\int_{0}^{u} \hat{a}_T(v)\, dv \biggr\}\,du,\quad
T>0.
\end{equation}

Throughout the paper we use the following notations:
\begin{eqnarray*}
&\displaystyle \beta_T^{(1)} (t)=\int_{0}^{t}
g_T \bigl(\xi_T(s) \bigr)\,ds, \qquad \displaystyle
\beta _T^{(2)}(t)=\int_{0}^{t}
g_T \bigl(\xi_T(s) \bigr)\,dW_T(s),&
\\
&\displaystyle I_T(t)=F_T \bigl(\xi_T(t)
\bigr)+\int_{0}^{t}g_T \bigl(\xi_T(s) \bigr)
\,dW_T(s),
\end{eqnarray*}
where $\xi_T$ and $W_T$ are related via equation~(\ref{o1}), $g_T (x)$ is a
family of measurable, locally bounded real-valued functions, $F_T (x)$ is a
family of continuous real-valued functions.

\begin{ozn}\label{oz2}
We say that a family of stochastic processes
$\zeta_T=\{\zeta_T(t), t\geq0\}$
weakly converges, as $T\to\infty$, to a process
$\zeta=\{\zeta(t), t\geq0\}$ if, for any $L>0$, the measures $\mu_T[0,L]$
generated by the processes $\zeta_T$ on the interval $[0,L]$ weakly
converge to the measure $\mu[0,L]$ generated by the process $\zeta$
considered on the interval $[0,L]$.
\end{ozn}

To study the weak convergence, as $T\to\infty$, of the processes
$I_T(t)\,{=}\,F_T(\xi_T(t))+\int_{0}^{t} g_T(\xi_T(s))\,dW_T(s)$, where
$\xi_T$ is the solution to stochastic differential equation~(\ref {o1}), we
suppose additionally that the drift coefficients satisfy the following
assumption:
there exists a family of measurable, locally bounded functions $\hat {a}_T
(x)$ such that for any $L>0$
\[
\lim_{T\to\infty} \int_{0}^{L}\sup_{x} \bigl\llvert
a_T(t,x)-\hat{a}_T (x) \bigr\rrvert \,dt= 0.
\tag{$A_0$}\label{A0}
\]

Note, that due to condition~(\ref{A0}), some results about solutions $\hat
{\xi}_T$ to homogeneous equations~(\ref{o2}), which are obtained in
\cite{paperL}, have been extended to solutions $\xi_T$ to inhomogeneous
equations~(\ref{o1}). Therefore, by analogy to the paper \cite{paperL},
consider equations~(\ref{o1}) from the class $K (G_T  )$.

\begin{ozn}\label{oz1}
The class of equations of the form~(\ref{o1}) will be denoted by $K(G_T)$,
if there exist families of functions $\hat{a}_T (x)$ and $G_T(x)$,
$x\in\R$, such that:
\begin{enumerate}
\item[1)] $\hat{a}_T (x)$ are measurable locally bounded real-valued functions,
satisfying condition~(\ref{A0});
\item[2)] $G_T(x)$ have continuous derivatives $G_T'(x)$ and locally integrable
second derivatives $G_T''(x)$ a.e. with respect to the Lebesgue measure
such that, for all $T>0$, $x\in\R$ and $ t\ge 0$, for some constant $C>0$ the
following inequalities hold:
\begin{align}
\biggl[ G_T'(x)a_T(t,x)+
\frac{1}{2}G_T''(x) \biggr]^2+
\bigl[ G_T'(x) \bigr]^2 &\leq C \bigl[1+
\bigl\llvert G_T(x)\bigr\rrvert ^2 \bigr], \notag
\\
\bigl\llvert G_T(x_0)\bigr\rrvert &\leq C;
\tag{$A_1$}\label{A1}
\end{align}

\item[3)] there exist constants $C>0$ and $\alpha>0$ such that, for all $x\in \R$,
\[
\bigl\llvert G_T(x)\bigr\rrvert \geq C \llvert x\rrvert
^{\alpha};
\]
\item[4)] there exist a bounded function $\psi(|x|)$ and a constant
$m\geq0$ such that $\psi (|x| )\to0$ as $|x|\to0$ and, for all
$x\in\R$ and $T>0$ and for any measurable bounded set $B$, the following
inequality holds:
\[
\biggl\llvert f_T'(x)\int_{0}^{x}
\frac{\chi_B(G_T(u) )}{f'_T(u)}\,du \biggr\rrvert \leq\psi \bigl(\lambda(B) \bigr) \bigl[1+
\llvert x\rrvert ^m \bigr], \tag{$A_2$}\label{A2}
\]
where $\chi_B(x)$ is the indicator function of a set $B$, $\lambda(B)$ is the
Lebesgue measure of $B$, $f_T'(x)$ is the derivative of the function $f_T(x)$
defined by equality~(\ref{o3}).
\end{enumerate}
\end{ozn}

Assume that, for certain locally bounded functions $ q_T (x) $ and any
constant \mbox{$N>0$}, the following condition holds:
\[
\lim_{T\to\infty}\sup_{\llvert x\rrvert \leq N} \biggl\llvert f_T'(x)
\int_{0}^{x}\frac{q_T(v)}{f'_T(v)}\, dv \biggr\rrvert =0.
\tag{$A_3$}\label{A3}
\]

%s3 ###
\section{Statement of the main results}

We are in position to obtain the main result (Theorem~\ref{th1}) concerning
the weak compactness of stochastic processes $\zeta_T=\{\zeta
_T(t)=G_T(\xi_T(t)), t\geq0\}$, and use it in further investigation of
asymptotic behavior of the solutions (Theorem~\ref{th2}) and some functionals
of the solutions (Theorems~\ref{th3}--\ref{th7}) to inhomogeneous
It\^{o} stochastic differential equations~(\ref{o1}).

\begin{thm}\label{th1}
Let $\xi_T$ be a solution to equation~(\ref{o1}) and let there exist a family
of continuous functions $G_T(x)$, $x\in\R$ for which the derivative
$G_T'(x)$ is continuous and the second derivative $G_T''(x)$ exists
a.e. with respect to the Lebesgue measure and is locally integrable. Let the
functions $G_T(x)$ satisfy assumption~(\ref{A1}), for all $T>0$, $t\geq0$,
$x\in\R$. Then the family of the processes
$\zeta_T=\{\zeta_T(t)=G_T(\xi_T(t)), t\geq0\}$ is weakly compact.
\end{thm}

\begin{thm}\label{th2}
Let $\xi_T$ be a solution to equation~(\ref{o1}) from the class $K (G_T)$
and\break $G_T(x_0)\to y_0$, as $T\to\infty$. Assume that there exist
measurable locally bounded functions $a_0(x)$ and $\sigma_0(x)$ such that:
%
%e4 ###
\begin{enumerate}
\item[1)] the functions
\begin{align*}
q_T^{(1)}(x)&=G_T'(x)\hat{a}_T(x)+\frac{1}{2}G_T''(x)-a_0
\bigl(G_T(x) \bigr),
\\
q_T^{(2)}(x)&= \bigl[G_T'(x)
\bigr]^2- \sigma_0^2 \bigl(G_T(x)
\bigr),
\end{align*}
satisfy assumption~(\ref{A3});\vadjust{\eject}
\item[2)] the It\^{o} equation
\begin{equation}
\zeta(t)=y_0+\int_{0}^{t}a_0 \bigl(
\zeta(s) \bigr)\,ds+\int_{0}^{t}\sigma_0 \bigl(
\zeta(s) \bigr)\,d\hat{W}(s)\label{o33}
\end{equation}
has a unique weak solution $(\zeta, \hat{W} )$.
\end{enumerate}
Then the stochastic processes $\zeta_T=G_T(\xi_T(t))$ weakly
converge, as $T\to\infty$, to the solution $\zeta$ to equation~(\ref{o33}).
\end{thm}

\begin{thm}\label{th3}
Let $\xi_T$ be a solution to equation~(\ref{o1}) from the class $K  (G_T )$ and let assumptions of Theorem~\ref{th2} hold. Assume that for
measurable and locally bounded functions $g_T(x)$ there exists measurable and
locally bounded function $g_0(x)$ such that the function
\[
q_T(x)=g_T(x)- g_0 \bigl(G_T(x)
\bigr)
\]
satisfies assumption~(\ref{A3}). Then the stochastic processes $ \beta _T^{(1)}
(t)=\int_{0}^{t} g_T(\xi_T(s))\,ds$ weakly converge, as $T\to\infty$,
to the process
\[
\beta^{(1)} (t)=\int_{0}^{t} g_0 \bigl(\zeta(s)
\bigr)\,ds,
\]
where $\zeta$ is a solution to equation~(\ref{o33}).
\end{thm}

\begin{thm}\label{th4}
Let $\xi_T$ be a solution to equation~(\ref{o1}) from the class $K  (G_T
 )$, and let the assumptions of Theorem~\ref{th2} hold. Assume that, for
measurable and locally bounded functions $g_T(x)$, there exists a measurable
locally bounded function $g_0(x)$ such that
\begin{align}
\biggl\llvert f_T'(x) \int_{0}^{x}
\frac{g_T(v)}{f'_T(v)}\, dv \biggr\rrvert \chi_{|x|\leq N}&\leq C_N,
\notag
\\
\lim_{T\to\infty}\sup_{|x|\leq N} \biggl\llvert f_T'(x)
\int_{0}^{x}\frac{g_T(v)}{f'_T(v)}\,dv-g_0
\bigl(G_T(x) \bigr)G_T'(x) \biggr\rrvert
&=0\tag{$A_4$}\label{A4}
\end{align}
for all $N>0$. Then the stochastic processes $ \beta_T^{(1)} (t)=\int_{0}^{t} g_T(\xi_T(s))\,ds$ weakly converge, as $T\to\infty$, to the
process
\[
\tilde{\beta}^{(1)} (t)=2 \biggl(\int_{y_0}^{\zeta(t)}
g_0(x)\, dx-\int_{0}^{t} g_0 \bigl(\zeta(s)
\bigr)\sigma_0 \bigl(\zeta(s) \bigr)\,d\hat {W}(s) \biggr),
\]
where $ (\zeta, \hat{W} )$ is a solution to equation~(\ref{o33}).
\end{thm}

\begin{thm}\label{th5}
Let $\xi_T$ be a solution to equation~(\ref{o1}) from the class $K  (G_T
 )$, and let the assumptions of Theorem~\ref{th2} hold. Suppose that the
functions $\hat{a}_T (x)$ satisfy assumption~(\ref{A3}). Assume that, for
measurable and locally bounded functions $g_T(x)$, there exist two constants
$c_0$ and $b_0$ such that for all $N>0$
\begin{eqnarray*}
&\displaystyle \biggl\llvert f_T'(x) \int_{0}^{x} \frac{g_T(v)}{f'_T(v)}\,dv \biggr\rrvert
\chi_{|x|\leq N}\leq C_N,&
\\
&\displaystyle \lim_{T\to\infty}\sup_{|x|\leq N} \biggl\llvert \int_{0}^{x}
\biggl[f'_T(u)\int_{0}^{u}
\frac{g_T(v)}{f'_T(v)} \,dv-c_0 \biggr]\,du \biggr\rrvert = 0,&
\end{eqnarray*}
and the functions
\[
Q_T(x)= \biggl[f'_T(x)\int_{0}^{x} \frac{g_T(v)}{f'_T(v)}\, dv-c_0
\biggr]^2-b_0^2
\]
satisfy assumption~(\ref{A3}). Then the stochastic processes
\[
\beta^{(1)}_T(t)=\int_{0}^{t}
g_T \bigl(\xi_T(s) \bigr)\,ds
\]
weakly converge, as $T\to\infty$, to the process $2b_0W(t)$, where $\{ W(t),
t\geq0\}$ is a Wiener process.
\end{thm}

\begin{thm}\label{th6}
Let $\xi_T$ be a solution to equation~(\ref{o1}) from the class $K  (G_T
 )$ and let assumptions of Theorem~\ref{th2} hold. Assume that, for
measurable and locally bounded functions $g_T(x)$, there exists a measurable
locally bounded function $g_0(x)$ such that the function
\[
q_T(x)= \bigl[g_T(x)- g_0
\bigl(G_T(x) \bigr)G_T'(x)
\bigr]^2
\]
satisfies assumption~(\ref{A3}). Then the stochastic processes
\[
\beta_T^{(2)} (t)=\int_{0}^{t}
g_T \bigl(\xi_T(s) \bigr)\,dW_T(s),
\]
where $\xi_T$ and $W_T$ are related via equation~(\ref{o1}), weakly converge,
as $T\to\infty$, to the process
\[
{\beta}^{(2)} (t)=\int_{0}^{t} g_0 \bigl(
\zeta(s) \bigr)\sigma_0 \bigl(\zeta (s) \bigr)\,d\hat{W}(s),
\]
where $ (\zeta, \hat{W} )$ is a solution to equation~(\ref{o33}).
\end{thm}

\begin{thm}\label{th7}
Let $\xi_T$ and $W_T$ be related via equation~(\ref{o1}) from the class
$K (G_T  )$ and let the assumptions of Theorem~\ref{th2} hold.
Assume that, for continuous functions $F_T(x)$ and locally bounded measurable
functions $g_T(x)$, there exist a continuous function $F_0(x)$ and locally
bounded measurable function $g_0(x)$ such that, for all $N>0$
\[
\lim_{T\to\infty}\sup_{|x|\leq N} \bigl\llvert F_T(x)-F_0
\bigl(G_T(x) \bigr) \bigr\rrvert =0,
\]
and let the functions $g_T(x)$ and $g_0(x)$ satisfy the assumptions of
Theorem~\ref{th6}. Then the stochastic processes
\[
I_T(t)=F_T \bigl(\xi_T(t) \bigr)+\int_{0}^{t}
g_T \bigl(\xi_T(s) \bigr)\,dW_T(s)
\]
weakly converge, as $T\to\infty$, to the process
\[
I_0(t)=F_0 \bigl(\zeta(t) \bigr)+\int_{0}^{t}
g_0 \bigl(\zeta(s) \bigr)\sigma_0 \bigl(\zeta (s)
\bigr) \,d\hat{W}(s),
\]
where $ (\zeta, \hat{W} )$ is a solution to equation~(\ref{o33}).
\end{thm}

%s4 ###
\section{Proof of the main results}

In the proofs of theorems, which are performed similarly to the proofs of the
corresponding theorems in \cite{paperL}, we emphasize the differences
associated with inhomogeneous equations. The proof of the Theorem~\ref{th2}
is given for a better understanding of the brief proofs of the other
theorems.

\begin{proof}[Proof of Theorem~\ref{th1}]
The functions $G_T(x)$ have continuous derivatives $G_T'(x)$ for all
$T>0$, their second derivatives $G_T''(x)$ exist a.e. with respect to
the Lebesgue measure and are locally integrable. Therefore (see
\cite{book12}, Chap. II, \S10), we can apply the It\^{o} formula to the
process $\zeta_T(t)=G_T(\xi_T(t))$, and with probability one, for all
$t\geq0$, we obtain
\begingroup
\abovedisplayskip=5.5pt
\belowdisplayskip=5.5pt
%e5 ###
\begin{equation}
\label{o4} \zeta_T (t)=G_T(x_0)+\int_{0}^{t}L_T
\bigl(\xi_T(s) \bigr)\,ds+\int_{0}^{t}G_T'
\bigl(\xi_T(s) \bigr)\,dW_T(s),
\end{equation}
where
\[
L_T(x)= G_T'(x)a_T(t,x)+
\frac{1}{2}G_T''(x).
\]

Let
%
%\[
%\chi_{N}(t)=\lleft\{ { %
%\begin{array}{l}
%1,\,\,\,\mathop{\sup}\limits_{0\leq s\leq t}  \llvert \zeta_T(s) \rrvert \leq N, \\
%0,\hspace{0.2cm} \mathop{\sup}\limits_{0\leq s\leq t}  \llvert \zeta
%_T(s) \rrvert >N. \\
%\end{array} %
%} \rright.
%\]
\[
\chi_{N}(t)=
\begin{cases}%{l}
1,\quad \sup_{0\leq s\leq t}  \llvert \zeta_T(s) \rrvert \leq N, \\
0,\quad \sup_{0\leq s\leq t}  \llvert \zeta_T(s) \rrvert >N.
\end{cases}
\]

It is clear that for $s\leq t$ we have $\chi_{N}(t)\chi_{N}(s)=\chi _{N}(t)$
with probability one. Thus, according to~(\ref{o4}), the following equality
holds with probability one:
%
%e6 ###
\begin{align}
\zeta_T (t)\chi_{N}(t)&=\zeta_T (0)
\chi_{N}(t) +\chi_{N}(t)\int_{0}^{t}L_T
\bigl(\xi_T(s) \bigr)\chi_{N}(s)\,ds\notag
\\
&\quad +\chi_{N}(t)\int_{0}^{t}G_T'
\bigl(\xi_T(s) \bigr)\chi_{N}(s)\,dW_T(s).\label{o6}
\end{align}

Hence, using condition~(\ref{A1}) and the properties of stochastic integrals, we
obtain that
\begin{align}
\M\zeta^2_T (t)\chi_{N}(t)&\leq 3 \biggl[\M
\zeta^2_T (0)\chi_{N}(t)+\M \biggl(\int_{0}^{t}L_T \bigl(\xi_T(s) \bigr)
\chi_{N}(s)\,ds \biggr)^2\notag
\\
&\quad +\M \biggl( \int_{0}^{t}G_T'
\bigl(\xi_T(s) \bigr)\chi_{N}(s)\,dW_T(s)
\biggr)^2 \biggr]\notag
\\
&\leq 3 \biggl[\M\zeta^2_T (0)\chi_{N}(t)+t
\int_{0}^{t}\M L_T^2 \bigl(
\xi_T(s) \bigr)\chi_{N}(s)\,ds\notag
\\
&\quad + \int_{0}^{t}\M \bigl[G_T'
\bigl(\xi_T(s) \bigr) \bigr]^2\chi_{N}(s)\,ds
\biggr]\notag
\\
&\leq 3 \biggl[C+ t\int_{0}^{t}C \bigl[1+\M\zeta^2_T
(s)\chi_{N}(s) \bigr]\,ds\notag
\\
&\quad +C\int_{0}^{t} \bigl[1+\M\zeta^2_T
(s)\chi_{N}(s) \bigr]\,ds \biggr]\notag
\\
&\leq C_L^{(1)}+C_L^{(2)}\int
_{0}^{t}\M \zeta^2_T (s)
\chi_{N}(s)\,ds,\label{o7}
\end{align}
\endgroup
where $C_L^{(1)}=3C(1+t+t^2)$, $C_L^{(2)}=3C(1+t)$, $C>0$ is a constant from
condition~(\ref{A1}), $0\leq t\leq L$.

Using the Gronwall--Bellman inequality, we conclude that there exists a
constant $K_L$, which is independent of $T$, and for $0\leq t\leq L$
\[
\M\zeta^2_T (t)\chi_{N}(t)\leq
K_L.
\]
Let $N\uparrow\infty$, then $\zeta^2_T (t)\chi_{N}(t)\uparrow\zeta ^2_T (t)$,
and we get the inequality
%
%e8 ###
\begin{equation}
\label{o8} \sup_{0\leq t\leq L}\M\zeta^2_T (t)\leq
K_L.
\end{equation}

Similarly to~(\ref{o7}), using~(\ref{o4}) and the inequality
\[
\M\sup_{0\leq t\leq L} \biggl[ \int_{0}^{t}G_T'
\bigl(\xi_T(s) \bigr)\,dW_T(s) \biggr]^2
\leq4 \int_{0}^{L}\M \bigl[G_T' \bigl(
\xi_T(s) \bigr) \bigr]^2\,ds,
\]
we conclude that
\begin{align*}
\M\sup_{0\leq t\leq L} \bigl\llvert \zeta_T (t) \bigr
\rrvert^2&\leq3 \Biggl[G_T^2(x_0)
+L\int_{0}^{L}C \bigl[1+\M\zeta^2_T(s)
\bigr] \,ds\,{+}\!\int_{0}^{L} C \bigl[1+\M
\zeta^2_T(s) \bigr] \, ds \! \Biggr]
\\
&\leq\tilde{C}_L^{(1)}+\tilde{C}_L^{(2)}
\int_{0}^{L}\M \bigl[\zeta _T (s)
\bigr]^2\,ds.
\end{align*}

Therefore, considering~(\ref{o8}), we obtain the inequality
%
%e9 ###
\begin{equation}
\label{o9} \M\sup_{0\leq t\leq L} \bigl\llvert \zeta_T (t)
\bigr\rrvert ^2\leq\tilde{K}_L
\end{equation}
for all $L>0$, where the constants $\tilde{K}_L$ are independent of $T$.

Using the inequalities for martingales
and for stochastic integrals (see \cite{book2}, Part~I, \S3, Theorem~6), we
obtain that
\begin{align*}
&\M\sup_{0\leq t\leq L} \Biggl\llvert \int_{0}^{t}G_T'
\bigl(\xi_T(s) \bigr)\chi_{N}(s)\,dW_T(s)
\Biggr\rrvert ^{2m}
\\
&\quad \leq \biggl(\frac{2m}{2m-1} \biggr)^{2m}\M \Biggl\llvert \int
_{0}^{L}G_T' \bigl(
\xi_T(s) \bigr)\chi_{N}(s)\,dW_T(s) \Biggr
\rrvert ^{2m}
\\
&\quad \leq \biggl(\frac{2m}{2m-1} \biggr)^{2m} \bigl[m(2m-1)
\bigr]^{m-1}L^{m-1}\int_{0}^{L}\M
\bigl[G_T' \bigl(\xi_T(s) \bigr)
\bigr]^{2m}\chi _{N}(s)\,ds,
\end{align*}
for any natural number $m$. Therefore, similarly to~(\ref{o9}) we have
inequality
%
%e10 ###
\begin{equation}
\label{o99} \M\sup_{0\leq t\leq L} \bigl\llvert \zeta_T
(t) \bigr\rrvert ^{2m}\leq{K}_{Lm}.
\end{equation}

Furthermore, for all $\alpha>0$ there exists $m\in\N$ such that $\alpha
\leq2m$ and, for random variable $\eta$, we have
\[
\M|\eta|^{\alpha}\leq1+\M|\eta|^{2m}.
\]

The last inequality together with~(\ref{o99}) implies that
%
%e11 ###
\begin{equation}
\label{o10} \M\sup_{0\leq t\leq L} \bigl\llvert \zeta_T (t)
\bigr\rrvert ^{\alpha}\leq{K}_{L\alpha}
\end{equation}
for all $\alpha>0$ and $L>0$, where the constants ${K}_{L\alpha}$ are
independent of $T$.

Since for $t_1<t_2\leq L$
\begin{align*}
&\M \bigl[\zeta_T(t_2)-\zeta_T(t_1)
\bigr]^4
\\
&\quad \leq8 \Biggl[\M \Biggl(\;\int_{t_1}^{t_2}L_T
\bigl(\xi_T(s) \bigr)\,ds \Biggr)^4+ \M \Biggl(\;\int
_{t_1}^{t_2}G_T' \bigl(
\xi_T(s) \bigr)\,dW_T(s) \Biggr)^4 \Biggr]
\\
&\quad \leq8 \Biggl[(t_2-t_1)^3\;\int
_{t_1}^{t_2}\M \bigl\llvert L_T \bigl(\xi
_T(s) \bigr) \bigr\rrvert^4\,ds+ 36(t_2-t_1)
\;\int_{t_1}^{t_2}\M \bigl[G_T'
\bigl(\xi_T(s) \bigr) \bigr]^4\,ds \Biggr],
\end{align*}
%
%then,
considering condition~(\ref{A1}) and inequality~(\ref{o10}), we get
%
%e12 ###
\begin{equation}
\label{o11} \M \bigl[\zeta_T(t_2)-
\zeta_T(t_1) \bigr]^4\leq
C_L|t_2-t_1|^2,
\end{equation}
where the constants $C_L$ are independent of $T$.

According to~(\ref{o8}) and~(\ref{o11}), we have convergences
%
%e13 ###
\begin{align}
\lim_{N\to\infty}\overline{\lim_{T \to\infty}} \sup_{0\leq t\leq L} \pr \bigl\{\bigl\llvert \zeta_T(t)\bigr\rrvert
>N \bigr\}&=0, \notag
\\
\lim_{h\to0}\overline{\lim_{T \to\infty}}\sup_{\substack{\llvert t_{1} -t_{2} \rrvert \le h \\ t_{i}\le L} } \pr \bigl\{\bigl\llvert \zeta_T(t_2)-
\zeta_T(t_1)\bigr\rrvert >\varepsilon \bigr
\}&=0\label{o12}
\end{align}
for any $L>0$, $\varepsilon>0$.

It means that we can apply Skorokhod's convergent subsequence principle (see
\cite{book6}, Chapter I, \S6) for the processes $\zeta_T(t)$ for all $0\leq
t\leq L$. According to this principle, given an arbitrary sequence
$T'_n\to\infty $, we can choose a subsequence \mbox{$T_n\to\infty$}, a probability
space $(\tilde\varOmega,\tilde\Im, \tilde \pr)$, and stochastic processes
$\tilde{\zeta}_{T_n}(t)$, ${\zeta}(t)$ defined on this space such that their
finite-dimensional distributions coincide with those of the processes
$\zeta_{T_n}(t)$, and, moreover,
$\tilde{\zeta}_{T_n}(t)\stackrel{\tilde{\pr}}{\to}{\zeta }(t)$, as
$T_n\to\infty$, for all $0\leq t\leq L$. The processes
$\tilde{\zeta}_{T_n}(t)$ and ${\zeta}(t)$ can be considered separable.

Using~(\ref{o11}), we have
\[
\M \bigl[\tilde{\zeta}_{T_n}(t_2)-\tilde{
\zeta}_{T_n}(t_1) \bigr]^4\leq C_L
\llvert t_2-t_1\rrvert ^2
\]
for all $0\leq t_1 \leq t_2\leq L$.

By Fatou's lemma,
\[
\M \bigl[\zeta(t_2)-\zeta(t_1) \bigr]^4\leq
C_L\llvert t_2-t_1\rrvert ^2.
\]
Thus, the processes $\tilde{\zeta}_{T_n}(t)$ and ${\zeta}(t)$ are continuous
with probability one. We have that finite-dimensional distributions of the
processes $\tilde {\zeta}_{T_n}(t)$ converge, as $T_n\to\infty$, to the
correspondent finite-dimensional distributions of the process ${\zeta}(t)$.
For a weak convergence of the processes ${\zeta}_{T_n}(t)$ it is sufficient
(see \cite{book11}, Chapter IX, \S2) to prove
%
%e14 ###
\begin{equation}
\label{o13} \lim_{h\to0}\overline{\lim_{T_n
\to\infty}}\pr \Bigl\{ \sup_{\substack{\llvert t_{1} -t_{2}
\rrvert \le h \\ t_{i} \le L} } \bigl\llvert \zeta_{T_n}(t_2)-
\zeta_{T_n}(t_1) \bigr\rrvert >\varepsilon \Bigr\}=0
\end{equation}
for any $L>0$, $\varepsilon>0$.

Due to inequalities (see \cite{book11}, Chapter IX, \S3)
\begin{align*}
&\pr \bigl\{\sup_{\substack{\llvert t_{1} -t_{2} \rrvert \le h \\t_{i} \le L} } \bigl\llvert \zeta_{T_n}(t_2)-
\zeta_{T_n}(t_1) \bigr\rrvert >\varepsilon \bigr\}
\\
&\quad \leq\sum_{kh\leq L}\pr \biggl\{ \sup_{kh\leq t\leq(k+1)h } \bigl\llvert \zeta_{T_n}(t)-\zeta_{T_n}(kh)
\bigr\rrvert > \frac{\varepsilon}{4} \biggr\}
\\
&\quad \leq \biggl(\frac{4}{\varepsilon} \biggr)^48\sum
_{kh\leq L} \Biggl\{\M\sup_{kh\leq t\leq(k+1)h } \Biggl(\;\int
_{kh}^{t}L_{T_n} \bigl(\xi_T(s)
\bigr)\,ds \Biggr)^4
\\
&\qquad +\M\sup_{kh\leq t\leq(k+1)h } \Biggl(\;\int_{kh}^{t}G_{T_n}'
\bigl(\xi_T(s) \bigr)\,dW_T(s) \Biggr)^4
\Biggr\}\leq \biggl(\frac{4}{\varepsilon} \biggr)^4K_L\sum
_{kh\leq L} \bigl(h^4+h^2 \bigr),
\end{align*}
where $K_L$ are independent of $T_n$, we obtain~(\ref{o13}). The proof of
Theorem~\ref{th1} is complete.
\end{proof}

\begin{proof}[Proof of Theorem~\ref{th2}]
Rewrite equality~(\ref{o4}) as
%
%e15 ###
\begin{equation}
\label{o14} \zeta_T (t)=G_T(x_0)+\int_{0}^{t}a_0
\bigl(\zeta_T(s) \bigr)\,ds+\eta _T(t)+
\alpha^{(0)}_T(t)+\alpha^{(1)}_T(t),
\end{equation}
where
\begin{align*}
\eta_T(t)&=\int_{0}^{t}G_T'
\bigl(\xi_T(s) \bigr)\,dW_T(s),
\\
\alpha^{(0)}_T(t)&=\int_{0}^{t}G_T'
\bigl(\xi_T(s) \bigr)\Delta a_T(s)\, ds,\quad\Delta
a_T(s)=a_T \bigl(s,\xi_T(s) \bigr)-
\hat{a}_T \bigl(\xi_T(s) \bigr),
\\
\alpha^{(1)}_T(t)&=\int_{0}^{t}q^{(1)}_T
\bigl(\xi_T(s) \bigr)\,ds,\quad q^{(1)}_T(x)=G_T'(x)
\hat{a}_T(x)+\frac{1}{2}G_T''(x)-a_0
\bigl(G_T(x) \bigr).
\end{align*}

The conditions~(\ref{A0}) and~(\ref{A1}), together with inequality~(\ref {o10}),
imply that
\begin{align}
\sup_{0\leq t\leq L} \bigl\llvert \alpha^{(0)}_T(t)
\bigr\rrvert &\leq\int_{0}^{L} \bigl\llvert
G_T' \bigl(\xi_T(s) \bigr) \bigr\rrvert
 \bigl\llvert \Delta a_T(s) \bigr\rrvert \,ds\notag
\\
&\leq \Bigl[C \Bigl(1+\sup_{0\leq s\leq L} \bigl\llvert\zeta_T(s)
\bigr\rrvert ^2 \Bigr) \Bigr]^{\frac{1}{2}}\int_{0}^{L}
\sup_{x} \bigl\llvert a_T(s,x)-\hat
{a}_T(x) \bigr\rrvert \,ds\stackrel{\pr} {\to} 0,\label{o15}
\end{align}
as $T\to\infty$ for any $L>0$.

The functions $q^{(1)}_T(x)$ satisfy conditions of Lemma~\ref{lm2}. Thus, for
any $L>0$
%
%e17 ###
\begin{equation}
\label{o16} \sup_{0\leq t\leq L} \bigl\llvert \alpha^{(1)}_T(t)
\bigr\rrvert \stackrel{\pr} {\to} 0,
\end{equation}
as $T\to\infty$.

It is clear that $\eta_T (t)$ is a family of continuous martingales with
quadratic characteristics
%
%e18 ###
\begin{equation}
\label{o17} \langle\eta_T\rangle(t)=\int_{0}^{t}
\bigl[G_T' \bigl(\xi_T(s) \bigr)
\bigr]^2\,ds=\int_{0}^{t}\sigma_0^2
\bigl(\zeta_T(s) \bigr)\,ds+\alpha^{(2)}_T(t),
\end{equation}
where
\[
\alpha^{(2)}_T(t)=\int_{0}^{t}q^{(2)}_T
\bigl(\xi_T(s) \bigr)\,ds,\quad q^{(2)}_T(x)=
\bigl( G_T'(x) \bigr)^2-
\sigma_0^2 \bigl(G_T(x) \bigr).
\]

The functions $q^{(2)}_T(x)$ satisfy conditions of Lemma~\ref{lm2}. Thus, for
any $L>0$
%
%e19 ###
\begin{equation}
\label{o18} \sup_{0\leq t\leq L} \bigl\llvert \alpha^{(2)}_T(t)
\bigr\rrvert \stackrel{\pr} {\to} 0,
\end{equation}
as $T\to\infty$.

According to Theorem~\ref{th1}, the family of the processes $\zeta_T(t)$ is
weakly compact. It is easy to see that compactness conditions~(\ref{o13}) are
fulfilled for the processes $\eta_T(t)$. Using convergences~(\ref{o15}),
(\ref{o16}),~(\ref{o18}), we have that relations~(\ref{o13}) hold for the
processes $\alpha^{(k)}_T(t)$, $k=0,1,2$ as well. It means that we can apply
Skorokhod's convergent subsequence principle (see \cite{book6}, Chapter I,
\S6) for the processes
\[
\bigl(\zeta_T(t),\ \eta_T(t),\ \alpha^{(k)}_T(t),\quad k=0,1,2 \bigr).
\]

According to this principle, given an arbitrary sequence $T'_n\to\infty $, we
can choose a subsequence $T_n\to\infty$, a probability space
$(\tilde\varOmega,\tilde\Im, \tilde \pr)$, and stochastic processes
\[
\bigl(\tilde{\zeta}_{T_n}(t),\ \tilde{\eta}_{T_n}(t),\ \tilde{\alpha }^{(k)}_{T_n}(t),\quad k=0,1,2 \bigr)
\]
defined on this space such that their finite-dimensional distributions
coincide with those of the processes
\[
\bigl(\zeta_{T_n}(t),\ \eta_{T_n}(t),\ \alpha^{(k)}_{T_n}(t),\quad k=0,1,2 \bigr),
\]
and, moreover,
\[
\tilde{\zeta}_{T_n}(t)\stackrel{\tilde{\pr}} {\to}\tilde{\zeta}(t),
\qquad \tilde{\eta}_{T_n}(t)\stackrel{\tilde{\pr}} {\to}\tilde{\eta}(t),
\qquad \tilde{\alpha}^{(k)}_{T_n}(t)\stackrel{\tilde{\pr}} {\to}
\tilde{\alpha }^{(k)}(t),\quad k=0,1,2,
\]
as $T_n\to\infty$, for all $0\leq t\leq L$, where $\tilde{\zeta}(t)$,
$\tilde{\eta}(t)$, $\tilde{\alpha}^{(k)}(t),k=0,1,2$ are some stochastic
processes.

Evidently, relations~(\ref{o15})--(\ref{o18}) imply that $\tilde{\alpha
}^{(k)}(t)\equiv0$, $k=0,1,2$ a.s. According to~(\ref{o11}), the processes
$\tilde{\zeta}(t)$ and $\tilde{\eta}(t)$ are continuous with probability one.
Moreover, applying Lemma~\ref{lm5} together with equalities~(\ref{o14}) and
(\ref{o17}), we obtain that
%
%e20 ###
\begin{align}
\label{o19} \tilde{\zeta}_{T_n} (t)&=G_{T_n}(x_0)+
\int_{0}^{t}a_0 \bigl(\tilde{\zeta
}_{T_n}(s) \bigr)\,ds+\tilde{\alpha}^{(0)}_{T_n}(t)+
\tilde{\alpha }^{(1)}_{T_n}(t)+\tilde{\eta}_{T_n}(t),
\\
\langle\tilde{\eta}_{T_n}\rangle(t)&=\int_{0}^{t}\sigma
_0^2 \bigl(\tilde{\zeta}_{T_n}(s) \bigr)\,ds+
\tilde{\alpha}^{(2)}_{T_n}(t),\notag
\end{align}
where
\[
\begin{array}{c}\tilde{\zeta}_{T_n}(t)\stackrel{\tilde{\pr}}{\to}\tilde
{\zeta}(t),\qquad\tilde{\eta}_{T_n}(t)\stackrel{\tilde{\pr}}{\to}\tilde
{\eta}(t),\qquad
\sup_{0\leq t\leq L} \bigl\llvert  \tilde{\alpha}^{(k)}_{T_n}(t) \bigr\rrvert \stackrel{\tilde{\pr}}{\to} 0, \quad  k=0,1,2,\\
\end{array} %
\]
as $T_n\to\infty$.

An analogue of convergence~(\ref{o13}) holds for the processes $\tilde
{\zeta}_{T_n}(t)$ and $\tilde{\eta}_{T_n}(t)$. Therefore, according to the
well-known result of Prokhorov (Lemma~1.11 in \cite{paper12}), we conclude
that for any $L>0$
%
%e21 ###
\begin{equation}
\label{o20} \sup_{0\leq t\leq L} \bigl\llvert \tilde{\zeta}_{T_n}(t)-
\tilde{\zeta }(t) \bigr\rrvert \stackrel{\tilde{\pr}} {\to} 0, \qquad\sup_{0\leq t\leq L}
\bigl\llvert \tilde{\eta}_{T_n}(t)-\tilde {\eta}(t) \bigr\rrvert
\stackrel{\tilde{\pr}} {\to} 0,
\end{equation}
as $T_n\to\infty$.

According to Lemma~\ref{lm3}, we can pass to the limit in~(\ref{o19}) and
obtain
\[
\tilde{\zeta} (t)=y_0+\int_{0}^{t}a_0
\bigl(\tilde{\zeta}(s) \bigr)\, ds+\tilde{\eta}(t),
\]
where $\tilde{\eta} (t)$ is the almost surely continuous martingale with the
quadratic characteristic
\[
\langle\tilde{\eta}\rangle(t)=\int_{0}^{t}\sigma_0^2
\bigl(\tilde {\zeta}(s) \bigr)\,ds.
\]

Now, it is well known, that the latter representation provides the existence
of a Wiener process $\hat{W}(t)$ such that
\[
\tilde{\eta}(t)=\int_{0}^{t}\sigma_0 \bigl(
\tilde{ \zeta}(s) \bigr)\,d\hat W(s).
\]

Thus, the process $ (\tilde{\zeta}, \hat{W} )$ satisfies equation
(\ref{o33}), and the processes $\tilde{\zeta}_{T_n}(t)$ weakly converge, as
$T_n\to\infty$, to the process $\tilde{\zeta}$. Since the sequence
$T^{'}_n\to\infty$ is arbitrary and since a solution to equation~(\ref{o33})
is weakly unique, the proof of the Theorem~\ref{th2} is complete.
\end{proof}

The proof of Theorems~\ref{th3}--\ref{th4}, \ref{th6}--\ref{th7} is performed
similarly to the proof of the corresponding theorems in \cite{paperL} with
some differences that we discuss below.

\begin{zau}\label{zth4}
The proof of Theorem~\ref{th4} differs from the proof of Theorem~2.3 in
\cite{paperL} by using other representation for the functional
$\beta_T^{(1)} (t)=\int_{0}^{t} g_T(\xi_T(s))\,ds$. In this case we
have
\[
{\beta}_{T}^{(1)} (t)= 2\int_{G_T(x_0)}^{\zeta_T(t)}g_0(u)
\,du -2\int_{0}^{t}g_0 \bigl({ \zeta}_{T}(s)
\bigr)\,d\eta_T(s)+\gamma ^{(1)}_T(t)-
\gamma^{(2)}_T(t)- \gamma^{(0)}_T(t),
\]
where
\[
\gamma^{(1)}_T(t)=\int_{x_0}^{\xi_T(t)}
\hat{q}_T(u)\,du,\qquad \gamma^{(2)}_T(t)=\int_{0}^{t}
\hat{q}_T \bigl({\xi}_{T}(s) \bigr)\,dW_T(s),
\]
\[
\gamma^{(0)}_T(t)= \int_{0}^{t}
\varPhi'_T \bigl(\xi_T(s) \bigr)
\bigl[a_T \bigl(s,\xi _T(s) \bigr)-\hat{a}_T
\bigl(\xi_T(s) \bigr) \bigr]\,ds,
\]
\[
\varPhi_T(x)=2\int_{0}^{x}f'_T(u)
\biggl(\int_{0}^{u}\frac
{g_T(v)}{f'_T(v)}\,dv \biggr)\,du,
\]
\[
\hat{q}_T(x)=\varPhi'_T(x)-2g_0
\bigl(G_T(x) \bigr)G_T'(x),
\]
$f_T'(x)$ is the derivative of the function $f_T(x)$ defined by equality
(\ref{o3}).

The latter representation differs from the corresponding representation in
\cite{paperL} by the term $\gamma^{(0)}_T(t)$. For any constants
$\varepsilon>0$, $N>0$ and $L>0$, we have the inequalities
\begin{align*}
\pr \Bigl\{\sup_{0\leq t\leq L} \bigl\llvert \gamma ^{(0)}_T(t)
\bigr\rrvert >\varepsilon \Bigr\}&\leq P_{\mathit{NT}} +\frac{2}{\varepsilon}\int
_{0}^{L}\M \bigl\llvert\varPhi'_T
\bigl(\xi_T(s) \bigr) \bigr\rrvert
\\
&\quad \times \bigl\llvert a_T \bigl(s,\xi_T(s)
\bigr)-\hat{a}_T \bigl(\xi_T(s) \bigr) \bigr\rrvert
\chi_{\llvert \xi_T(s)\rrvert \leq N}\,ds
\\
&\leq P_{\mathit{NT}}+\frac{2}{\varepsilon}C_N\int
_{0}^{L}\sup_{x}
\bigl[a_T(s,x)-\hat{a}_T(x) \bigr]\,ds,
\end{align*}
where $P_{\mathit{NT}}=\pr\{\sup_{0\leq t\leq L} |\xi
_{T}(t)|>N\}$. Using condition~3) from Definition~\ref{oz1} and
inequality~(\ref{o10}), we obtain the convergence $\lim_{N\to\infty}\;\overline{\lim_{T \to\infty}}\;P_{\mathit{NT}}=0$.
Using the assumptions of Theorem~\ref{th4}, we conclude that
%
%e22 ###
\begin{equation}
\label{o0} \sup_{0\leq t\leq L} \bigl\llvert \gamma^{(0)}_T(t)
\bigr\rrvert \stackrel{\pr} {\to} 0
\end{equation}
for any $L>0$, as $T\to\infty$. The rest of the proof of Theorem~\ref {th4}
is the same as that of Theorem~2.3 in~\cite{paperL}.
\end{zau}

The proofs of Theorem~\ref{th3} and Theorems~\ref{th6}--\ref{th7} are
literally the same as those of Theorem~2.2 and Theorems~2.4--2.5 from
\cite{paperL}.

\begin{proof}[Proof of Theorem~\ref{th5}]
For the functional $\beta_T^{(1)} (t)=\int_{0}^{t} g_T(\xi
_T(s))\,ds$, with probability one, for all $t\geq0$, we have the
representation
\[
\begin{array}{l}
{\beta}_{T}^{(1)} (t)=2c_0\int_{0}^{t}\hat{a}_T(\xi_T(s))\,
ds+\gamma_T(t)-\eta^{(1)}_T(t)-\gamma^{(0)}_T(t)+\gamma^{(3)}_T(t),\\
\end{array} %
\]
where
\begin{align*}
\gamma_T(t)&=2\int_{x_0}^{\xi_T(t)}
\Biggl[f'_T(u)\int_{0}^{u}
\frac{g_T(v)}{f'_T(v)}\,dv-c_0 \Biggr]\,du,
\\
\eta^{(1)}_T(t)&=\int_{0}^{t}
\bigl[\varPhi'_T \bigl(\xi_T(s)
\bigr)-2c_0 \bigr]\,dW_T(s),
\\
\gamma^{(3)}_T(t)&= 2c_0\int
_{0}^{t} \bigl[a_T \bigl(s,
\xi_T(s) \bigr)-\hat {a}_T \bigl(\xi_T(s)
\bigr) \bigr]\,ds,
\end{align*}
$\gamma^{(0)}_T(t)$ and $\varPhi_T(x)$ are defined in Remark~\ref{zth4}.\vadjust{\eject}

The functions $\hat{a}_T(x)$ satisfy condition~(\ref{A3}). Thus, using Lemma~\ref{lm2}, we conclude that for any $L>0$
\[
\sup_{0\leq t\leq L} \biggl\llvert \int_{0}^{t}
\hat{a}_T \bigl(\xi_T(s) \bigr)\, ds \biggr\rrvert
\stackrel{\pr} {\to} 0,
\]
as $T\to\infty$.

For any constants $\varepsilon>0$, $N>0$ and $L>0$, we have the inequalities
\begin{align*}
\pr \Bigl\{\sup_{0\leq t\leq L} \bigl\llvert \gamma_T(t) \bigr
\rrvert >\varepsilon \Bigr\}&\leq P_{\mathit{NT}}+\frac{1}{\varepsilon}\M\sup_{0\leq t\leq L} \biggl\llvert \int_{x_0}^{\xi_T(t)} \bigl[
\varPhi'_T(u)-2c_0 \bigr]\,du \biggr
\rrvert \chi_{\{\llvert \xi_T(t)\rrvert \leq N\}}
\\
&\leq P_{\mathit{NT}}+\frac{2}{\varepsilon}N \sup_{|x|\leq
N} \biggl\llvert \;\;
\int_{x_0}^{x} \biggl[f'_T(u) \int_{0}^{u}\frac
{g_T(v)}{f'_T(v)}\,dv-c_0 \biggr]\,du
\biggr\rrvert ,
\end{align*}
where $P_{\mathit{NT}}$ is the same as that in Remark~\ref{zth4}. Using the latter
inequality and the assumptions of Theorem~\ref{th5}, we conclude that
\[
\sup_{0\leq t\leq L} \bigl\llvert \gamma_T(t) \bigr\rrvert
\stackrel{\pr} {\to} 0,
\]
as $T\to\infty$.

Since the term $\gamma^{(0)}_T(t)$ is the same as that in Remark~\ref {zth4},
we have~(\ref{o0}).

The inequality
\[
\sup_{0\leq t\leq L} \bigl\llvert \gamma^{(3)}_T(t)
\bigr\rrvert \leq2c_0\int_{0}^{L}\sup_{x}
\bigl[a_T(s,x)-\hat{a}_T(x) \bigr]\,ds
\]
implies that for any $L>0$
\[
\sup_{0\leq t\leq L} \bigl\llvert \gamma^{(3)}_T(t)
\bigr\rrvert \stackrel{\pr} {\to} 0,
\]
as $T\to\infty$.

Thus, we have for any $L>0$
\[
\sup_{0\leq t\leq L} \bigl\llvert {\beta}_{T}^{(1)}
(t)+\eta^{(1)}_T(t) \bigr\rrvert \stackrel{\pr} {\to} 0,
\]
as $T\to\infty$.

It is clear that $\eta^{(1)}_T(t)$ is the almost surely continuous martingale
with the quadratic characteristic
\[
\bigl\langle{\eta^{(1)}_T} \bigr\rangle(t)=4b_0^2t+
\int_{0}^{t}q_T \bigl(\xi _T(s) \bigr)
\,ds,
\]
where $q_T(x)= [\varPhi'_T(x)-2c_0 ]^2-4b_0^2$. The functions $q_T(x)$
satisfy condition~(\ref{A3}). Thus, using Lemma~\ref{lm2}, we conclude that for
any $L>0$
\[
\sup_{0\leq t\leq L} \bigl\llvert \bigl\langle{\eta^{(1)}_T}
\bigr\rangle (t)-4b_0^2t \bigr\rrvert \stackrel{\pr} {
\to} 0,
\]
as $T\to\infty$.

Then, using a random change of time (see \cite{book2}, Part I, \S4), we
obtain $\eta^{(1)}_T(t)=W_T^{*} (\langle{\eta^{(1)}_T}\rangle
(t) )$, where $W_T^{*}(t)$ is a Wiener process. The same arguments as
those used to get (9) in \cite{paper88} yield that
\[
\sup_{0\leq t\leq L} \bigl\llvert {\beta}_{T}^{(1)}
(t)-W_T^{*} \bigl(4b_0^2t \bigr)
\bigr\rrvert \stackrel{\pr} {\to} 0,
\]
as $T\to\infty$. Thus, the processes ${\beta}_{T}^{(1)} (t)$ weakly converge,
as $T\to\infty$, to the process $2b_0 W(t)$.
\end{proof}

%s5 ###
\section{Auxiliary results}

\begin{lem}\label{lm1}
Let $\xi_T$ be a solution to equation~(\ref{o1}) from the class $K  (G_T
 )$. Then, for any $N>0$, $L>0$ and any Borel set $B\subset
 [-N;N ]$, there exists a constant $C_L$ such that
\[
\int_{0}^{L}\pr \bigl\{G_T \bigl(
\xi_T(s) \bigr)\in B \bigr\}\,ds\leq C_L \psi \bigl(
\lambda(B) \bigr),
\]
where $\lambda(B)$ is the Lebesgue measure of the set $B$, $\psi
(|x| )$ is a certain bounded function satisfying $\psi
(|x| )\to0$ as $|x|\to0$.
\end{lem}

\begin{proof}
Consider the function
\[
\varPhi_T(x)=2\int_{0}^{x}f'_T(u)
\biggl(\int_{0}^{u}\frac{\chi
_B (G_T(v) )}{f'_T(v)}\,dv \biggr)\,du.
\]

The function $\varPhi_T(x)$ is continuous, the derivative $\varPhi'_T(x)$ of this
function is continuous and the second derivative $\varPhi''_T(x)$ exists a.e.
with respect to the Lebesgue measure and is locally integrable. Therefore, we
can apply the It\^{o} formula to the process $\varPhi_T(\xi_T(t))$, where
$\xi_T(t)$ is a solution to equation~(\ref{o1}).

Furthermore,
\[
\varPhi'_T(x)\hat{a}_T(x)+
\frac{1}{2}\varPhi''_T(x)=
\chi_B(x),
\]
a.e. with respect to the Lebesgue measure. Using the It\^{o} formula and the
latter equality, we conclude that
\[
\int_{0}^{L}\chi_B \bigl( \zeta_T(s)
\bigr)\,ds=\varPhi_T \bigl(\xi _T(L) \bigr)-
\varPhi_T(x_0)- \int_{0}^{L}
\varPhi'_T \bigl(\xi_T(s) \bigr)
\,dW_T(s)-\alpha_T(L),
\]
with probability one for all $t\geq0$, where $\zeta_T (t)=G_T(\xi _T(t))$,
\[
\alpha_T(t)= \int_{0}^{t}\varPhi'_T
\bigl(\xi_T(s) \bigr) \bigl[a_T \bigl(s,\xi
_T(s) \bigr)-\hat{a}_T \bigl(\xi_T(s)
\bigr) \bigr]\,ds.
\]
Hence, using the properties of stochastic integrals, we obtain that
%
%e23 ###
\begin{equation}
\label{l2} \int_{0}^{L}\pr \bigl\{\zeta_T (s)\in B
\bigr\}\,ds=\M \bigl[\varPhi_T \bigl(\xi _T(L) \bigr)-
\varPhi_T(x_0) \bigr]-\M\alpha_T(L).
\end{equation}

According to condition~(\ref{A2}), inequalities $|G_T(x)|\geq C |x|^{\alpha }$,
$C>0$, $\alpha>0$ and~(\ref{o10}), we have
\[
\bigl\llvert \M \bigl[\varPhi_T \bigl(\xi_T(L) \bigr)-
\varPhi_T(x_0) \bigr] \bigr\rrvert \leq
C^{(1)}_L \psi \bigl(\lambda(B) \bigr),
\]
for a certain constant $C^{(1)}_L$. Condition~(\ref{A0}) implies that
\[
\bigl\llvert \M\alpha_T(L) \bigr\rrvert \leq C^{(2)}_L
\psi \bigl(\lambda(B) \bigr),
\]
for a certain constant $C^{(2)}_L$. Here the function
$\psi (\lambda(B) )$ is from condition~(\ref{A2}). The latter
inequalities and equality~(\ref{l2}) prove Lemma~\ref{lm1}.
\end{proof}

\begin{lem}\label{lm2} Let $\xi_T$ be a solution to equation~(\ref{o1})
from the class $K (G_T  )$. If, for measurable locally bounded
functions $q_T(x)$, condition~(\ref{A3}) holds, then, for any $L>0$,
\[
\sup_{0\leq t\leq L} \biggl\llvert \int_{0}^{t}
q_T \bigl(\xi_T(s) \bigr)\,ds \biggr\rrvert \stackrel{
\pr} {\to} 0,
\]
as $T\to\infty$.
\end{lem}

\begin{proof}
Consider the function
\[
\varPhi_T(x)=2\int_{0}^{x}f'_T(u)
\biggl(\int_{0}^{u}\frac
{q_T(v)}{f'_T(v)}\,dv \biggr)\,du.
\]

The function $\varPhi_T(x)$ is continuous, the derivative $\varPhi'_T(x)$ of this
function is continuous and the second derivative $\varPhi''_T(x)$ exists a.e.
with respect to the Lebesgue measure and is locally integrable. Therefore, we
can apply the It\^{o} formula to the process $\varPhi_T(\xi_T(t))$, where
$\xi_T(t)$ is a solution to equation~(\ref{o1}).

Furthermore,
\[
\varPhi'_T(x)\hat{a}_T(x)+
\frac{1}{2}\varPhi''_T(x)=q_T(x),
\]
a. e. with respect to the Lebesgue measure. Using the latter equality, we
conclude that with probability one for all $t\geq0$
%
%e24 ###
\begin{equation}
\label{l3} \int_{0}^{t}q_T \bigl(
\xi_T(s) \bigr)\,ds=\varPhi_T \bigl(
\xi_T(t) \bigr)- \varPhi_T(x_0)- \int_{0}^{t} \varPhi'_T \bigl(
\xi_T(s) \bigr) \,dW_T(s)-\alpha_T(t),
\end{equation}
where
\[
\alpha_T(t)= \int_{0}^{t}\varPhi'_T
\bigl(\xi_T(s) \bigr) \bigl[a_T \bigl(s,\xi
_T(s) \bigr)-\hat{a}_T \bigl(\xi_T(s)
\bigr) \bigr]\,ds.
\]

For any constants $\varepsilon>0$, $N>0$ and $L>0$, we have
\begin{align*}
\pr \Bigl\{\sup_{0\leq t\leq L} \bigl\llvert \alpha_T(t)
\bigr\rrvert >\varepsilon \Bigr\}&\leq P_{\mathit{NT}} +\frac{2}{\varepsilon}\sup
_{|x|\leq N}f'_T(x) \biggl\llvert \int_{0}^{x}
\frac{q_T(v)}{f'_T(v)}\,dv \biggr\rrvert
\\
&\quad \times\int_{0}^{L}\sup
_{x} \bigl[a_T(s,x)-\hat{a}_T(x)
\bigr]\,ds,
\end{align*}
where $P_{\mathit{NT}}=\pr \{\sup_{0\leq t\leq L}  \llvert \xi
_{T}(t) \rrvert >N \}$.\vadjust{\eject}

The same arguments as those used in \cite{paperL} (see the proof of Lemma~4.2) and the assumptions of Lemma~\ref{lm2} yield that
\begin{eqnarray*}
&\displaystyle \sup_{0\leq t\leq L} \bigl\llvert \alpha_T(t)
\bigr\rrvert\stackrel{\pr} {\to} 0,&
\\
&\displaystyle \sup_{0\leq t\leq L} \bigl\llvert \varPhi_T
\bigl(\xi_T(t) \bigr)-\varPhi_T(x_0) \bigr
\rrvert\stackrel {\pr} {\to} 0,&
\\
&\displaystyle \sup_{0\leq t\leq L} \Biggl\llvert \int
_{0}^{t}\varPhi'_T \bigl(
\xi_T(s) \bigr)\,dW_T(s) \Biggr\rrvert \stackrel{\pr} {
\to} 0,&
\end{eqnarray*}
as $T\to\infty$. Thus, equality~(\ref{l3}) implies the statement of Lemma~\ref{lm2}.
\end{proof}

The statements and the proofs of the following lemmas are the same as those
of the corresponding lemmas from \cite{paperL}.

\begin{lem}\label{lm5}
Let $\xi_T$ be a solution to equation~(\ref{o1}) belonging to the class
$K (G_T  )$, and let the stochastic process $ (\zeta_T,\; \eta_T
 )$, with $\zeta_T (t)=G_T(\xi_T(t)) $ and $\eta_T(t)=\int_{0}^{t}G_T'(\xi_T(s))\,dW_T(s)$ be stochastically equivalent to
the process $ (\tilde{\zeta}_T,\;\tilde{\eta}_T  )$. Then the
process
\[
\int_{0}^{t}g \bigl(\zeta_T(s) \bigr)\,ds+ \int_{0}^{t}q \bigl(\zeta_T(s) \bigr)\, d
\eta_T(s),
\]
where $g(x)$ and $q(x)$ are measurable locally bounded functions, is
stochastically equivalent to the process
\[
\int_{0}^{t}g \bigl(\tilde{\zeta}_T(s) \bigr) \,ds+
\int_{0}^{t}q \bigl(\tilde {\zeta}_T(s) \bigr)\,d
\tilde{\eta}_T(s).
\]
\end{lem}

\begin{lem}\label{lm3}
Let $\xi_T$ be a solution to equation~(\ref{o1}) from the class $K  (G_T
 )$, and let $\zeta_T (t)=G_T(\xi_T(t))\stackrel{\pr}{\to }\zeta(t)$ as
$T\to\infty$. Then for any measurable locally bounded function $g(x)$, we
have the convergence
\[
\sup_{0\leq t\leq L} \biggl\llvert \int_{0}^{t} g
\bigl( \zeta_T(s) \bigr)\,ds-\int_{0}^{t} g \bigl(
\zeta(s) \bigr)\,ds \biggr\rrvert \stackrel{\pr} {\to} 0,
\]
as $T\to\infty$ for any constant $L>0$.
\end{lem}

%s6 ###
\section{Examples}

We denote by $b_T$ the family of such constants that $b_T>1$ and
$b_T\uparrow\infty$ as $T\to\infty$.
\begin{ex}\label{ex1}

Consider equation~(\ref{o1}) with the drift coefficient with nonregular
dependence on the parameter $T$ of the form
\[
a_T(t,x)=b_T^{\gamma}\cos(x b_T) +
\frac{tb_T}{1+t^2b_T^2} \sin \bigl((x-1)b_T \bigr), \quad 0\leq\gamma<1.
\]

The family of measurable locally bounded real-valued functions $\hat {a}_T
(x)=\break b_T^{\gamma}\cos(x b_T)$ satisfies condition~1) from Definition~\ref{oz1}:
for any $L>0$
\[
\lim_{T\to\infty} \int_{0}^{L}\sup_{x} \bigl\llvert
a_T(t,x)-\hat {a}_T (x) \bigr\rrvert \,dt \leq \lim_{T\to\infty} \int_{0}^{L} \frac{tb_T}{1+t^2b_T^2} \,dt = 0.
\]

The rest of conditions from Definition~\ref{oz1} are fulfilled for the family
of functions
\[
G_T(x)=f_T(x)=\int_{0}^{x}\exp \biggl\{-2
\int_{0}^{u}\hat {a}_T(v)\,dv \biggr\}\,du, \quad T>0.
\]

Since $f_T'(x)=\exp \{-2\frac{b_T^{\gamma}}{b_T}\sin(x b_T)  \}$,
then there exist two constants $c_0$ and $\delta_0$ such that, for all
$x\in\R$, we have $0<\delta_0\leq f_T'(x)\leq c_0$. Taking into account that
$G_T(x)=\int_{0}^{x}f'_T(v)\,dv$, we obtain $G_T'(x)\hat
{a}_T(x)+\frac{1}{2}\,G_T''(x)\equiv0$. Therefore, the conditions
of Definition~\ref{oz1} %2.2
are fulfilled as follows:
\begin{enumerate}[leftmargin=55pt]
\item[condition 2)]
\begin{align*}
& \biggl[ G_T'(x)a_T(t,x)+
\frac{1}{2}G_T''(x) \biggr]^2+
\bigl[G_T'(x) \bigr]^2
\\
&\quad = \biggl[ G_T'(x) \frac{tb_T}{1+t^2b_T^2} \sin
\bigl((x-1)b_T \bigr) \biggr]^2+ \bigl[
G_T'(x) \bigr]^2\leq 2
\bigl[G_T'(x) \bigr]^2
\\
&\quad \leq2c_0^2 \leq2c_0^2
\bigl[1+ \bigl\llvert G_T(x)\bigr\rrvert ^2 \bigr];
\\
&\bigl\llvert G_T(x_0)\bigr\rrvert = \biggl\llvert
\int_{0}^{x_0}f'_T(v) \,dv \biggr
\rrvert \leq c_0\cdot|x_0|= C;
\end{align*}

\item[condition 3)]
\[
\bigl\llvert G_T(x)\bigr\rrvert = \Biggl\llvert \int
_{0}^{x}f'_T(v) \,dv
\Biggr\rrvert \geq C |x|^{\alpha} \quad \mbox{with } C=\delta_0,\ \alpha=1;
\]

\item[condition 4)]
\begin{align*}
&\biggl\llvert \int_{0}^{x} f_T'(u) \biggl(
\int_{0}^{u}\frac{\chi_B (G_T(v) )}{f'_T(v)}\,dv \biggr)\,du \biggr\rrvert\\
&\quad \leq\frac{C_0}{\delta_0} \biggl\llvert \int_{0}^{x} \!\!\!\int _{0}^{u}\chi_B \bigl(G_T(v) \bigr)\,dv\,du \biggr\rrvert\,{\leq}\, C_1 \lambda(B) |x| \,{\leq}\,\psi \bigl(\lambda(B) \bigr) \bigl[1\,{+}\,|x|^m \bigr]
\end{align*}
with $\psi (|x| )= C_1 |x|$, $m=1$.
\end{enumerate}

Thus, equation~(\ref{o1}) belongs to the class $K (G_T  )$.
According to Theorem~\ref{th1}, the family of the processes
$\zeta_T(t)=G_T (\xi_T(t) )$ is weakly compact. We can find the form
of the limit process using Theorem~\ref{th2} with $a_0(x)\equiv0$,
$\sigma_0(x)\equiv1$. According to Theorem~\ref{th2}, the stochastic
processes $\zeta_T(t)$ weakly converge, as $T\to\infty$, to the solution
$\zeta(t)$ to equation~(\ref{o33}) and the limit process is
$\zeta(t)=x_0+\hat{W}(t)$, where $\hat{W}(t)$ is a Wiener process.
\end{ex}

\begin{ex}\label{ex2}
Let the conditions of Example~\ref{ex1} hold. For the family of functions
\[
g_T(x)=\frac{ b_T^{\gamma}}{1+b_T^2 x^2 }, \quad 0\leq\gamma<1,
\]
the assumptions of Theorem~\ref{th3} hold with $g_0(x)\equiv0$.

According to Theorem~\ref{th3}, the stochastic processes
\[
\beta_T^{(1)} (t)=\int_{0}^{t}
g_T \bigl(\xi_T(s) \bigr)\,ds=\int_{0}^{t}
\frac{ b_T^{\gamma}}{1+b_T^2 \xi_T^2(s) }\,ds, \quad 0\leq\gamma<1
\]
weakly converge, as $T\to\infty$, to the process $ \beta^{(1)} (t)\equiv0$.
\end{ex}

\section*{Acknowledgments}

The authors are deeply grateful to the referees for careful reading and
valuable comments and suggestions, which helped to improve essentially the
quality of the article.

\end{document}